\documentclass[12pt, reqno]{amsart}
\usepackage[utf8]{inputenc}
\usepackage[english]{babel}
 
\usepackage[square,sort,comma,numbers]{natbib}

\usepackage[all,cmtip]{xy}

\usepackage[dvipsnames]{xcolor}
\usepackage{float}

\usepackage{graphicx}
\usepackage{tikz}
\usetikzlibrary{arrows}
\usetikzlibrary{matrix,arrows,decorations.pathmorphing}
\usepackage{pinlabel}

\usepackage{amsmath}
\usepackage{amssymb}
\usepackage{amsthm}

\usepackage{stmaryrd}
\usepackage[all]{xy}
\usepackage[mathcal]{eucal}
\usepackage{verbatim}
\usepackage[top=1in,bottom=1in,left=1in,right=1in]{geometry} 
\usepackage{wrapfig}
\usepackage{lipsum}
\usepackage[all]{xy}
\theoremstyle{plain}
\newtheorem{thm}{Theorem}[section]

\newtheorem{lem}[thm]{Lemma}

\newtheorem{cor}[thm]{Corollary}

\newtheorem{qu}[thm]{Problem}
\theoremstyle{definition}
\newtheorem{defn}[thm]{Definition}

\theoremstyle{remark}
\newtheorem*{rem}{Remark}

\newcommand{\nc}{\newcommand}
\nc{\dmo}{\DeclareMathOperator}

\DeclareMathOperator{\Homeo}{Homeo}
\DeclareMathOperator{\Stab}{Stab}

\DeclareMathOperator{\GL}{GL}

\DeclareMathOperator{\DD}{\mathbb{D}}
\DeclareMathOperator{\RR}{\mathbb{R}}

\DeclareMathOperator{\Mod}{Mod}

\DeclareMathOperator{\Hull}{Hull}

\newcommand{\R}{\mathbb{R}}
\newcommand{\Z}{\mathbb{Z}}
\newcommand{\pair}[1]{\langle #1 \rangle}

\nc{\para}[1]{\medskip\noindent\textbf{#1.}}
\title[Global fixed points of mapping class group actions]{\boldmath Global fixed points of mapping class group actions and a theorem of Markovic}
\author{Lei Chen and Nick Salter}
\date{February 26, 2020}
\address{LC: Department of Mathematics, California Institute of Technology, Pasadena, CA 91125}
\address{NS: Department of Mathematics, Columbia University, New York, NY 10027}
\email{chenlei@caltech.edu}
\email{nks@math.columbia.edu}

\thanks{NS is supported by NSF Award No. DMS-1703181.}

\begin{document}
 \bibliographystyle{alpha}
\maketitle
\vspace{-0.2in}
\begin{abstract}
We give a short and elementary proof of the non-realizability of the mapping class group via homeomorphisms. This was originally established by Markovic, resolving a conjecture of Thurston. With the tools established in this paper, we also obtain some rigidity results for actions of the mapping class group on Euclidean spaces. 
\end{abstract}

\section{introduction}
In this paper, we discuss a new strategy to study Nielsen's realization problem. We will use this to give a new and very simple proof of the main result of \cite{Mark}. 

Let $\Sigma_g$ be the surface of genus $g$ and $\Homeo_+(\Sigma_g)$ the group of orientation-preserving homeomorphisms of $\Sigma_g$. Denote by $\Mod(\Sigma_g):=\pi_0(\Homeo_+(\Sigma_g))$ the mapping class group of $\Sigma_g$. The {\em Nielsen realization problem} asks whether the natural projection
\[
p_g:\Homeo_+(\Sigma_g)\to \Mod(\Sigma_g)
\]
has a group-theoretic section (this particular formulation is attributed to Thurston; c.f. \cite[Problem 2.6]{kirby}). The realization problem can also be posed for other regularities and for arbitrary subgroups of $\Mod(\Sigma_g)$, and there is a rich literature on the many variants that arise. We refer the reader for the survey paper \cite{MT} for further discussion. 

In \cite{Mark}, Markovic shows that $p_g$ has no section for $g>5$. The proof is very involved and uses many dynamical tools. The main result of this note gives an elementary proof of his result in the optimal range $g \ge 2$.
\begin{thm}\label{Nielsen}
For $g\ge2$, the projection $p_g$ has no sections.
\end{thm}

We obtain this as a consequence of a rigidity theorem for actions of a closely related group. Let $\Sigma_{g,1}$ be the surface of genus $g$ with one marked point and let $\Homeo_+(\Sigma_{g,1})$ denote the group of orientation-preserving homeomorphisms of $\Sigma_{g,1}$ that fix the marked point. Define $\Mod(\Sigma_{g,1}):=\pi_0(\Homeo_+(\Sigma_{g,1}))$ to be the ``pointed mapping class group'' of $\Sigma_{g,1}$. 

\begin{thm}\label{R2}
For $g \ge 2$, any nontrivial action of $\Mod(\Sigma_{g,1})$ on $\RR^2$ by homeomorphisms has a global fixed point.
\end{thm}

We now explain how Theorem \ref{R2} implies Theorem \ref{Nielsen}.
\begin{proof}
Let $\Homeo_+(\DD^2)^{\pi_1(S_g)}$ denote the group of $\pi_1(\Sigma_g)$-equivariant orientation-preserving homeomorphisms of the universal cover $\DD^2$ of $\Sigma_g$. This is compatible with the Birman exact sequence for the mapping class group, realizing $\Homeo_+(\DD^2)^{\pi_1(S_g)}$ as the pullback of $\Homeo_+(\Sigma_g)$ and $\Mod(\Sigma_{g,1})$ along $\Mod(\Sigma_g)$:
\[
\xymatrix{
1 \ar[r]  & \pi_1(\Sigma_g)\ar[r]\ar@{=}[d] & \Homeo_+(\DD^2)^{\pi_1(S_g)} \ar[r]\ar[d]^{p_g'} &\Homeo_+(\Sigma_g)\ar[r]\ar[d]^{p_g}& 1 \\
1 \ar[r] &\pi_1(\Sigma_g)\ar[r]  &\Mod(\Sigma_{g,1}) \ar[r] &\Mod(\Sigma_g)\ar[r] &1 }
\]
By the universal property of pullbacks, a section of $p_g$ gives rise to a section of $p_g'$; any such section $p_g'$ must realize $\pi_1(\Sigma_g)$ as the group of deck transformations of $\Sigma_g$.

By Theorem \ref{R2}, the action of $\Mod(\Sigma_{g,1})$ on $\DD^2 \cong \mathbb{R}^2$ via a section of $p_g'$ has a global fixed point, which contradicts the fact that deck transformations act freely.
\end{proof}

The same set of ideas also leads to a rigidity theorem for mapping class group actions on $\R^3$ in the regime $g \ge 4$. 
\begin{thm}\label{R3}
For $g \ge 4$, any nontrivial continuous action of $\Mod(\Sigma_{g,1})$ on $\RR^3$ has a globally-invariant line.
\end{thm}
\begin{cor}\label{R23}
For $g \ge 4$, there is no action of $\Mod(\Sigma_{g,1})$ on $\RR^2$ and $\RR^3$ by $C^1$ diffeomorphisms.
\end{cor}

All of the above results are an easy consequence of the following structural result for $\Mod(\Sigma_{g,1})$. For an element $f$ of a group $G$, we write $C(f)$ to denote the centralizer of $f$ in $G$. Also recall that a homeomorphism $\iota$ is said to be {\em hyperelliptic} if $\iota$ has order $2$ and has exactly $2g+2$ fixed points. A mapping class is said to be hyperelliptic if it admits a hyperelliptic representative. 
\begin{thm}\label{main}
For $g\ge2$, there exists an order $6$ element $\alpha_g\in \Mod(\Sigma_{g,1})$ such that 
\[
\Gamma := \pair{C(\alpha_g^2), C(\alpha_g^3)}
\]
is the full mapping class group:
\[
\Gamma = \Mod(\Sigma_{g,1}).
\]
If $g \ne 3$, then $\alpha_g$ can be constructed so that $\alpha_g^3$ is not hyperelliptic.
\end{thm}

\begin{qu}
If $\beta \in \Mod(\Sigma_{g,1})$ is a torsion element of order divisible by two distinct primes $p, q$, then there is no {\em a priori} obstruction for $\Mod(\Sigma_{g,1})$ to be generated by $C(\beta^p)$ and $C(\beta^q)$. Does the conclusion of Theorem \ref{main} hold for any torsion element $\beta$ with order not a prime power?
\end{qu}

We now explain how Theorem \ref{main} implies Theorem \ref{R2}. In Section \ref{section:R3}, we show how Theorem \ref{main} also implies Theorem \ref{R3} and Corollary \ref{R23} via a short argument in local Smith theory.
\begin{proof}
Since $\Mod(\Sigma_{g,1})$ is perfect, it has no nontrivial maps to $\mathbb{Z}/2$. Thus any action on $\mathbb{R}^n$ is orientation-preserving.  Let $\alpha_g$ be the symmetry of Theorem \ref{main}. Any continuous action of the finite-order element $\alpha_g$ on $\mathbb{R}^2$ has a unique fixed point $O$ (see \cite{CK} and \cite{Ker2}). Since this fixed point is unique, both $C(\alpha_g^2)$ and $C(\alpha_g^3)$ fix $O$. By Theorem \ref{main}, $\Mod(\Sigma_{g,1})$ fixes $O$, showing Theorem \ref{R2}.
\end{proof}

In the remainder of this note we construct the symmetry $\alpha_g$ and establish the properties claimed in Theorem \ref{main}. In Section \ref{section:models} we describe our models of $\Sigma_{g,1}$ equipped with the symmetry $\alpha_g$; we use a different model for each residue class $g \pmod 3$. In Section \ref{section:chords}, we discuss a special class of curves and subsurfaces on these model subsurfaces that feature in the proof of Theorem \ref{main}. We carry out the proof of Theorem \ref{main} in Section \ref{section:proof}. Finally, we deduce Theorem \ref{R3} and Corollary \ref{R23} from Theorem \ref{main} in Section \ref{section:R3}.

\para{Acknowledgements}
The authors would like to thank Vlad Markovic for helpful discussions.

\section{Models}\label{section:models}
The aim of this section is to construct, for each $g \ge 2$, a certain symmetry $\alpha_g \in \Mod(\Sigma_{g,1})$ of order $6$. The proof of Theorem \ref{main} in Section \ref{section:proof} requires the existence of certain configurations of symmetric curves which are easy to find only for certain conjugacy classes of order-$6$ elements of $\Mod(\Sigma_{g,1})$; this is why we must take care in constructing our symmetries. The Riemann--Hurwitz formula (c.f. Lemma \ref{RH}) implies that different constructions are necessary for each residue class of $g \pmod 3$. In order to give as uniform a presentation as possible, we represent each ``model surface'' as a disk with pairs of boundary segments identified; the rules for edge identification are specified by the data of a ``monodromy tuple'' to be presented below (see the table in \eqref{tuple}). 

For $g \ne 3$, the symmetries we use (and the corresponding symmetric surfaces) are depicted in Figure \ref{figure:models}. The case $g = 3$ requires special consideration; the model we use is shown in Figure \ref{figure:genus3model}. The discussion leading up to Figure \ref{figure:models} is not absolutely required to make sense of Figure \ref{figure:models}, but is included so as to help orient the reader.

\subsection{Branched covers} The symmetries we construct are realized as deck transformations of $\Z/6\Z$-branched covers of $S^2$. Here we recall the basic topological theory of branched coverings. Fix a group $G$ and surfaces $X$ and $Y$; we also fix the {\em branch locus} $B \subset Y$, a finite set of points. A branched covering $f: X \to Y$ with covering group $G$ branched over $B$ is then specified by a surjective homomorphism $\rho: \pi_1(Y \setminus B) \to G$. The preimage $f^{-1}(B)$ is the {\em ramification set} and the elements are {\em ramification points}. A point $x \in X$ is ramified if and only if $\Stab_G(x)$ is nontrivial; in this case, the {\em order} of $x$ is defined to be the order of $\Stab_G(x)$. 

When $Y = S^2$ is a sphere, this can be further combinatorialized. Enumerate $B = \{b_1, \dots, b_n\}$, and choose an identification
\[
\pi_1(S^2 \setminus B) \cong \pair{a_1, \dots, a_n \mid a_1\dots a_n = 1};
\] 
here each $a_i$ runs from a basepoint $p \in S^2$ to a small loop around $b_i$. The {\em local monodromy} at $b_i$ is the corresponding element $\rho(a_i) \in G$. Without loss of generality we can assume that each $\rho(a_i) \ne 1$. The {\em monodromy vector} is the associated tuple $(\rho(a_1), \dots, \rho(a_n)) \in G^n$. Note that necessarily $\rho(a_1) \dots \rho(a_n) = 1$, and conversely, any such $n$-tuple gives rise to a branched $G$-cover.

For the purposes of this paper we will only be concerned with the case $G = \Z/6\Z$, and we adopt some further notation special to this situation. With the branch set $B \subset S^2$ fixed, we observe that each branch point $b_i$ has corresponding order $|\rho(a_i)| \in \{2,3,6\}$. Define $p$ (resp. $q$ or $r$) as the number of points of order $6$ (resp. $3$ or $2$). We define the {\em branching vector} as the tuple $(p,q,r)$. The lemma below records the Riemann--Hurwitz formula specialized to this setting.

\begin{lem}\label{RH}
Let $f: \Sigma_g \to S^2$ be a $\Z/6\Z$-branched covering with branching vector $(p,q,r)$. Then
\[
5p + 4q + 3r = 10+2g.
\]
\end{lem}

\para{Monodromy tuples} As a final specialization, we can shorten our notation for the monodromy vector at the cost of possibly re-ordering the elements of $B$. Suppose that $1 \in \Z/6\Z$ appears $a$ times, $2 \in \Z/6\Z$ appears $b$ times, etc. Up to a re-ordering of $B$, this data can be captured in the {\em monodromy tuple}. To make the computation of the associated $p,q,r$ more transparent, we order the elements of $\Z/6\Z$ according to their group-theoretic order. Thus a monodromy tuple is a symbol of the following form:
\begin{equation}\label{equation:tuple}
1^a\ 5^{p-a}\ 2^b\ 4^{q-b}\ 3^r.
\end{equation}

\subsection{The model surfaces} The elements $\alpha_g$ of Theorem \ref{main} will be constructed as deck transformations associated to regular $\Z/6\Z$ covers of $S^2$ as in the previous subsection. We will require different constructions for the three different residue classes $g \pmod 3$ and a special construction for $g = 3$. Below, we specify $k \ge 0$. As the final column shows, for $g \ne 3$, the power $\alpha_g^3$ has strictly fewer than $2g+2$ fixed points and hence is not hyperelliptic.

\begin{equation}\label{tuple}
\begin{array}{|c|c|c|c|}
\hline
g		& (p,q,r)		& \text{tuple}     & \text{number of branched points of $\alpha_g^3$}  
\\ \hline
2+3k		& (2,1,2k)		& 1^2\ 4\ 3^{2k}		&      2k\times 3+2<2g+2\\ 
3 		& (2,0,2)		& 1\ 5\ 3^{2} 		&	8=2g+2 \\
3 + 3(k+1)& (3,1,2k+1)	& 1^2\ 5\ 2\ 3^{2k+1} & (2k+1) \times 3 + 3 < 2g+2 \\ 
4 + 3k	& (3, 0,2k+1)	& 1^3\ 3^{2k+1}  	&(2k+1)\times 3+3<2g+2 \\ \hline
\end{array}
\end{equation}

\begin{figure}[]
\labellist
\small
\pinlabel $g\equiv2\pmod3$ at 100 250
\pinlabel $1$ [bl] at 158.73 388.31
\pinlabel $2$ [br] at 56.69 399.65
\pinlabel $3$ [br] at 0.00 320.28
\pinlabel $4$ [tr] at 36.85 229.58
\pinlabel $5$ [bl] at 198.41 315.28
\pinlabel $6$ [tl] at 195.57 299.78
\pinlabel $1$ [br] at 36.85 388.31
\pinlabel $2$ [tr] at 0.00 300.78
\pinlabel $3$ [tr] at 59.52 215.41
\pinlabel $4$ [tl] at 158.73 226.75
\pinlabel $5$ [tl] at 136.05 215.41
\pinlabel $6$ [bl] at 138.88 399.65
\pinlabel $x$ [tr] at 141.72 385.48
\pinlabel $g\equiv0\pmod3$ at 220 40
\pinlabel $1$ [bl] at 277.77 175.73
\pinlabel $2$ [br] at 178.57 189.90
\pinlabel $3$ [br] at 119.04 110.54
\pinlabel $4$ [tr] at 155.89 17.01
\pinlabel $5$ [tl] at 260.76 5.67
\pinlabel $6$ [tl] at 317.45 85.03
\pinlabel $1$ [br] at 155.89 178.57
\pinlabel $2$ [tr] at 119.04 85.03
\pinlabel $3$ [tr] at 178.57 5.67
\pinlabel $4$ [tl] at 277.77 17.01
\pinlabel $5$ [bl] at 317.45 110.54
\pinlabel $6$ [bl] at 257.93 189.90
\pinlabel $x$ [tr] at 263.60 175.73
\pinlabel $g\equiv1\pmod3$ at 340 250
\pinlabel $1$ [bl] at 397.77 385.73
\pinlabel $2$ [br] at 298.57 394.90
\pinlabel $3$ [br] at 239.04 315.54
\pinlabel $4$ [tr] at 275.89 232.01
\pinlabel $5$ [tl] at 380.76 219.67
\pinlabel $6$ [tl] at 437.45 297.03
\pinlabel $1$ [br] at 275.89 383.57
\pinlabel $2$ [tr] at 239.04 297.03
\pinlabel $3$ [tr] at 298.57 217.67
\pinlabel $4$ [tl] at 397.77 229.01
\pinlabel $5$ [bl] at 437.45 315.54
\pinlabel $6$ [bl] at 377.93 394.90
\pinlabel $x$ [tr] at 383.60 390.73

\tiny
\pinlabel $1'$ [b] at 121.88 402.48
\pinlabel $2'$ [tl] at 167.23 235.25
\pinlabel $2'$ [tr] at 5.67 269.27
\pinlabel $4'$ [br] at 28.34 379.81
\pinlabel $5'$ [bl] at 195.57 331.62
\pinlabel $5'$ [tr] at 85.03 209.74
\pinlabel $1'$ [tl] at 175.73 246.59
\pinlabel $3'$ [tr] at 0 283.44
\pinlabel $3'$ [b] at 110.54 405.15
\pinlabel $4'$ [bl] at 189.90 345.79
\pinlabel $6'$ [tr] at 70.86 212.58
\pinlabel $6'$ [br] at 19.84 368.47
\pinlabel $1'$ [b] at 240.92 195.57
\pinlabel $6'$ [br] at 144.55 167.23
\pinlabel $5'$ [tr] at 121.88 70.86
\pinlabel $4'$ [tl] at 192.74 2.83
\pinlabel $3'$ [tl] at 289.11 28.34
\pinlabel $2'$ [bl] at 314.62 121.88
\pinlabel $5'$ [br] at 136.05 153.06
\pinlabel $4'$ [tr] at 127.55 56.85
\pinlabel $3'$ [t] at 209.74 0.00
\pinlabel $2'$ [tl] at 297.61 42.52
\pinlabel $1'$ [bl] at 308.95 141.72
\pinlabel $6'$ [b] at 226.75 195.57
\pinlabel $1''$ [b] at 215.92 195.57
\pinlabel $3''$ [b] at 201.75 195.57
\pinlabel $4''$ [br] at 134.55 140.23
\pinlabel $6''$ [br] at 131.05 130.06
\pinlabel $5''$ [tl] at 216.74 2.83
\pinlabel $2''$ [tl] at 305.11 55.34
\pinlabel $6''$ [bl] at 299.62 146.88
\pinlabel $1''$ [tr] at 142.55 38.85
\pinlabel $4''$ [t] at 234.74 5.00
\pinlabel $3''$ [tl] at 307.61 66.52
\pinlabel $5''$ [bl] at 291.95 158.72
\pinlabel $2''$ [tr] at 135.88 47.86

\pinlabel $1'$ [b] at 360.92 405.57
\pinlabel $2'$ [br] at 264.55 377.23
\pinlabel $3'$ [tr] at 241.88 280.86
\pinlabel $4'$ [tl] at 307.74 212.83
\pinlabel $5'$ [tl] at 409.11 238.34
\pinlabel $6'$ [bl] at 434.62 331.88
\pinlabel $1'$ [br] at 256.05 363.06
\pinlabel $2'$ [tr] at 250.55 266.85
\pinlabel $3'$ [t] at 329.74 210.00
\pinlabel $4'$ [tl] at 417.61 252.52
\pinlabel $5'$ [bl] at 428.95 351.72
\pinlabel $6'$ [b] at 346.75 405.57

\endlabellist
\includegraphics[scale=1]{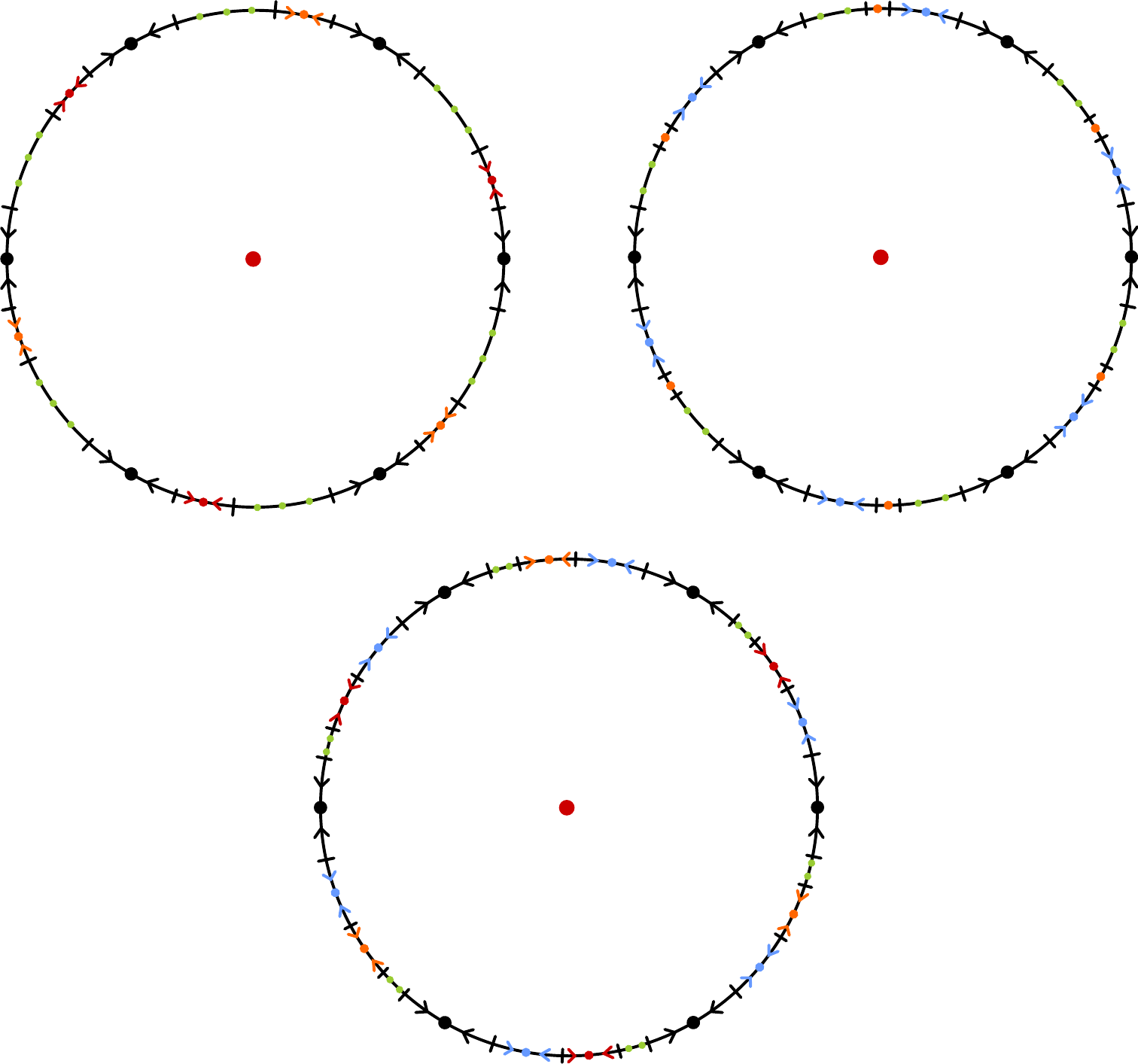}
\caption{The model surfaces for $g \ne 3$: each is constructed by taking the disk $D$ and identifying the specified edge segments of $\partial D$. In each case, $\alpha_g$ is represented as a rotation by $2\pi/6$ about the center of the disk. We adopt the convention that we do not specify labeling data for points where the local monodromy is of order $2$; in each such case, the point and its corresponding segment of $\partial D$ is identified with the segment directly opposite on $D$. This is true in particular for the $k \ge 0$ points in green in each sector, which correspond to the portion of the monodromy tuple of the form $3^{2k}$. }
\label{figure:models}
\end{figure}
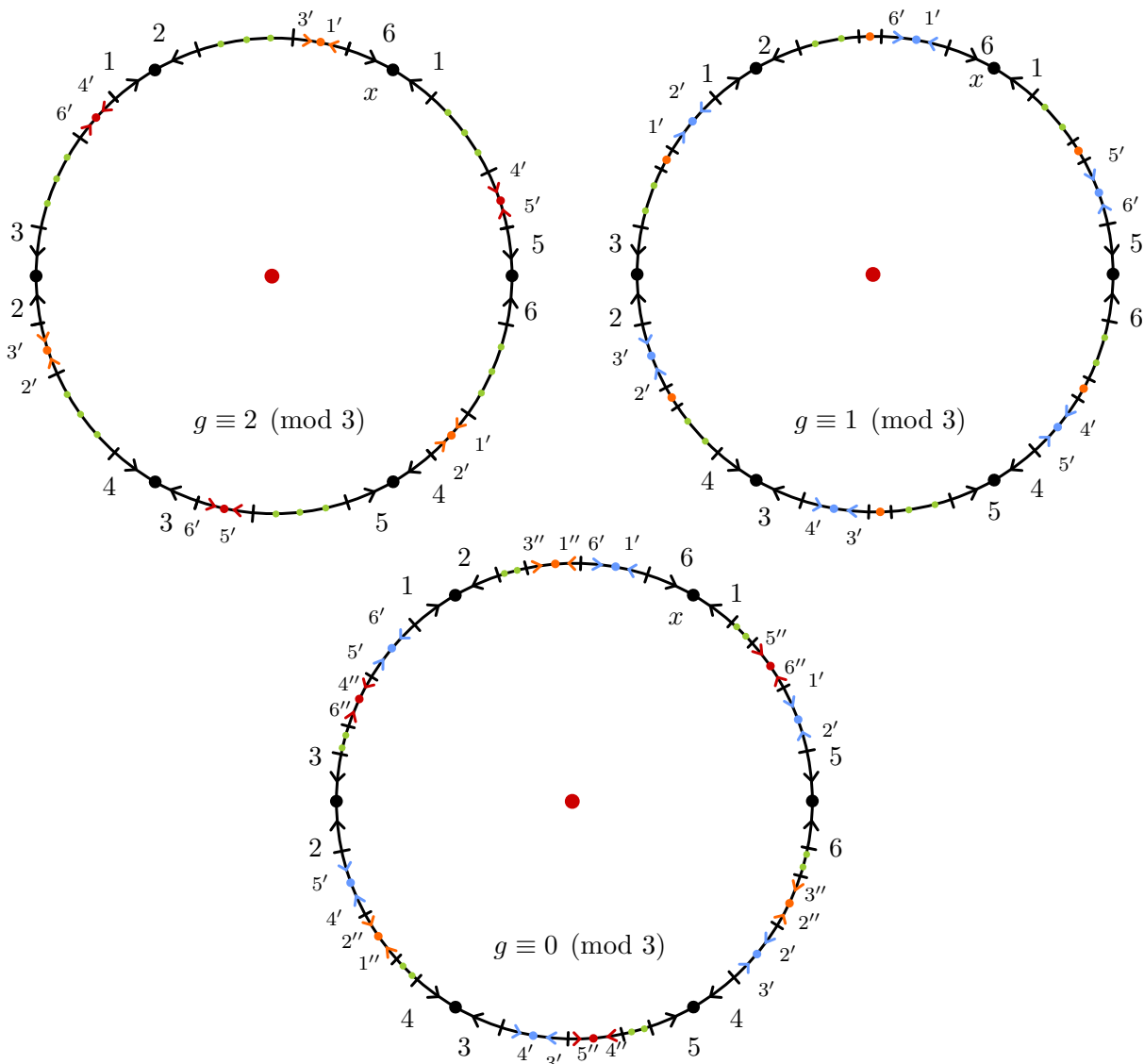
\newpage

To represent $\Sigma_{g,1}$ with its symmetry $\alpha_g$ as in \eqref{tuple}, we adopt the models shown in Figure \ref{figure:models}, where $\Sigma_{g,1}$ is given as a (marked) disk $D$ with edge identifications. See Figure \ref{figure:models} and its caption for a detailed discussion. We emphasize that the marked point $x \in \Sigma_{g,1}$ is {\em not} the fixed point at the center, but rather one of the fixed points on $\partial D$ (labeled, and drawn with a heavy dot). The model surface for $g = 3$ is given in Figure \ref{figure:genus3model}.

For the purpose of later discussion, we observe here a simple property of this construction.

\begin{defn}[Edge type]\label{definition:edgetype}
Let $e$ be an edge of $\partial D$. By construction, exactly one endpoint of $e$ is a ramification point. We say that $e$ is {\em type $p$} (resp. {\em type $q$, type $r$}) if this ramification point has order $6$ (resp. $3$, $2$). 
\end{defn}

\section{Chords and convexity} \label{section:chords}

In this section we develop some language for discussing a special class of curves and subsurfaces on the model surfaces. This is based around an ad-hoc identification of the disk $D$ with the {\em hyperbolic} disk $\mathbb D^2$. We will find it convenient to consider representatives for curves on $\Sigma_{g,1}$ as geodesics on $\mathbb D^2$, and especially to consider the notion of convexity in $\mathbb D^2$. We emphasize here that we are using $\mathbb D^2$ in a nonstandard way: $\mathbb D^2$ is {\em not} playing the role of the universal cover of $\Sigma_{g,1}$. Rather, we are viewing $\Sigma_{g,1}$ as a {\em topological quotient} of $\mathbb D^2$ under a set of identifications of portions of $\partial \mathbb D^2$. The geometry of $\mathbb D^2$ will provide us with a convenient framework in which to prove Theorem \ref{main}. 

\subsection{Chordal curves}
The first special structure inherited from imposing the hyperbolic metric on $D$ is a privileged (finite) set of simple closed curves: those that can be represented as single geodesics on $\mathbb D^2$. 

\begin{defn}[Chordal curve]
Let $\Sigma_{g,1}$ be given for $g \ge 2$, and let $D$ be the associated disk as shown in Figure \ref{figure:models}; we identify $D$ with the hyperbolic disk $\mathbb D^2$. A {\em chordal curve} is a simple closed curve $c \subset \Sigma_{g,1}$ that can be represented as a single geodesic on $\mathbb D^2$. A chordal curve is {\em basic} if its endpoints can be taken to lie on the {\em interiors} of the identified portions of $\partial D$. The {\em type} of a basic chordal curve is defined to be the type of the corresponding edge of $\partial \mathbb D^2$ in the sense of Definition \ref{definition:edgetype}. 
\end{defn}
See Figure \ref{figure:chords} for some examples and non-examples of chordal curves. \\

\begin{figure}[ht]
\labellist
\small
\pinlabel $a$ at 140 155
\pinlabel $b$ at 55 90
\pinlabel $c$ at 105 40
\pinlabel $d$ at 300 155
\pinlabel $d$ at 343 40
\pinlabel $e$ at 347 155
\pinlabel $x$ at 377 187
\endlabellist
\includegraphics[scale=1]{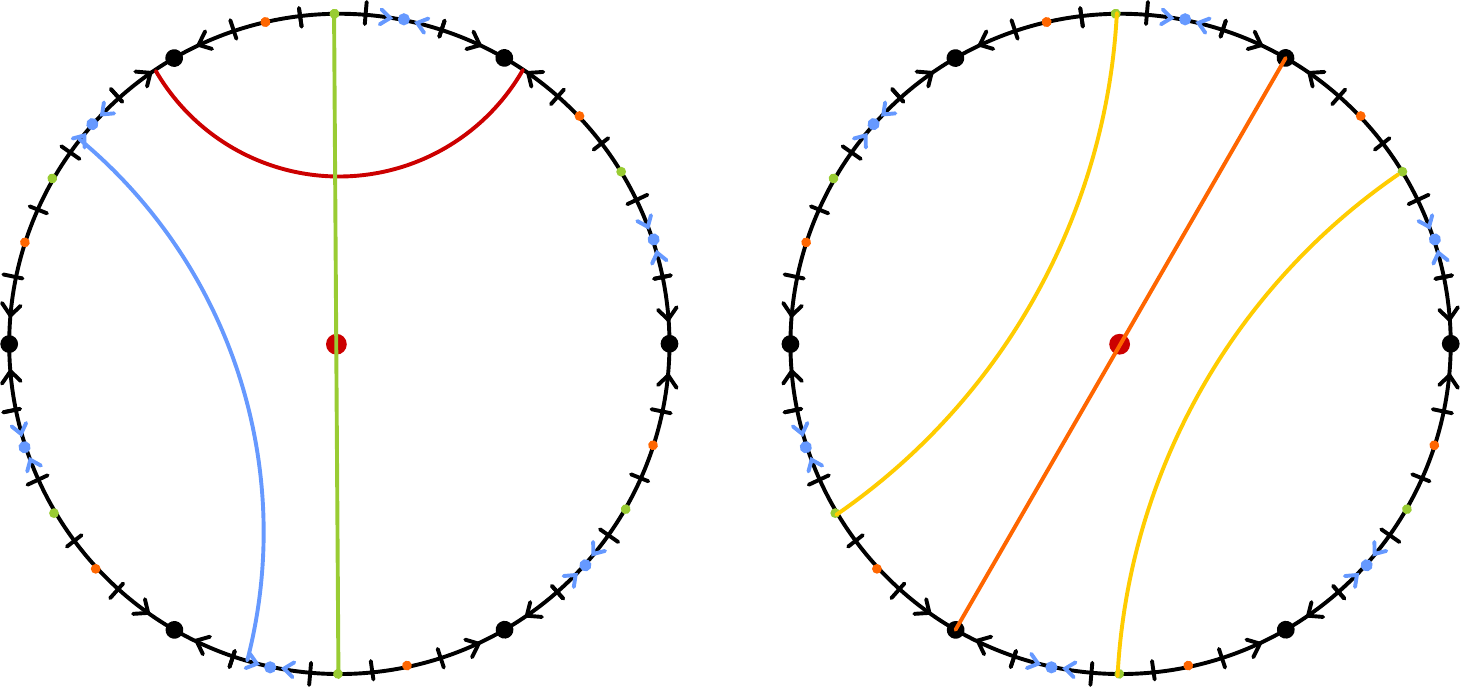}
\caption{At left, three chordal curves $a,b,c$ on $\Sigma_{5,1}$ of types $p,q,r$, respectively. At right, two curves $d,e$ that are not chordal. $d$ is not chordal because it cannot be represented as a single segment on $D$, and $e$ is not chordal because it passes through the marked point $x$.}
\label{figure:chords}
\end{figure}
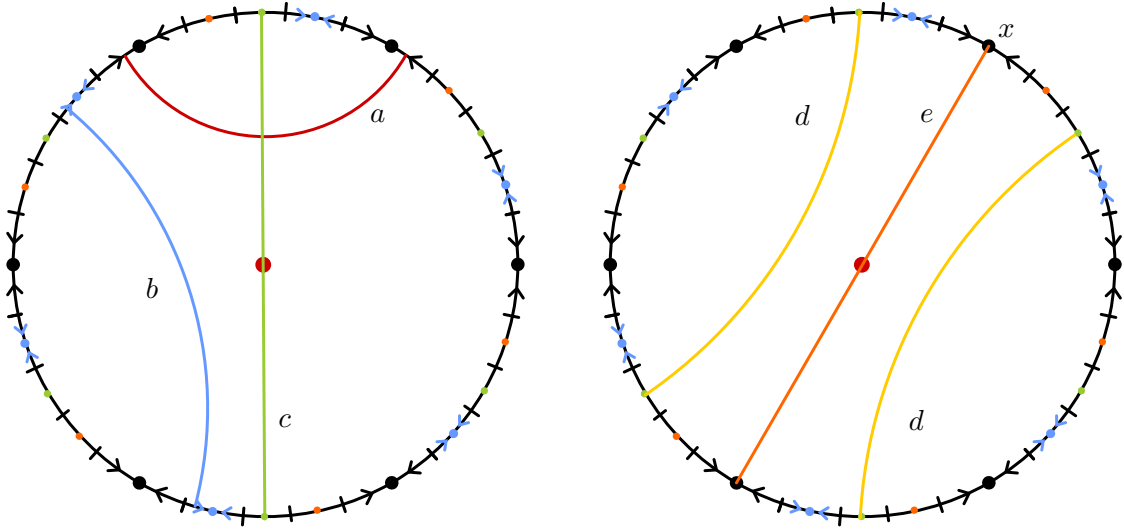

\begin{figure}[ht]
\labellist
\small
\pinlabel $c$ at 115 145
\pinlabel $\alpha_5^2(c)$ at 60 100
\pinlabel $\alpha_5^4(c)$ at 140 100
\pinlabel $d$ at 105 40
\endlabellist
\includegraphics[scale=1]{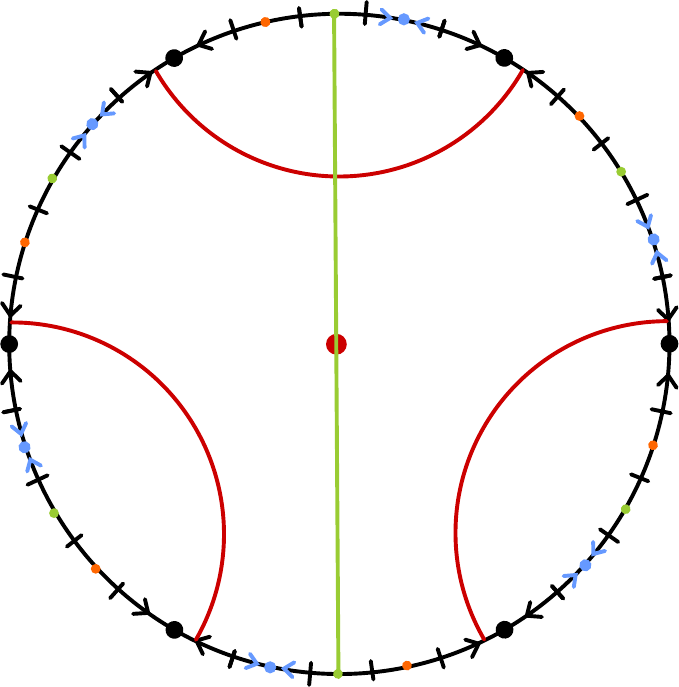}
\caption{The curves $c, \alpha_g^2(c), \alpha_g^4(c), d$ of Lemma \ref{lemma:commtrick}, illustrated for $g = 5$. }
\label{figure:commtrick}
\end{figure}
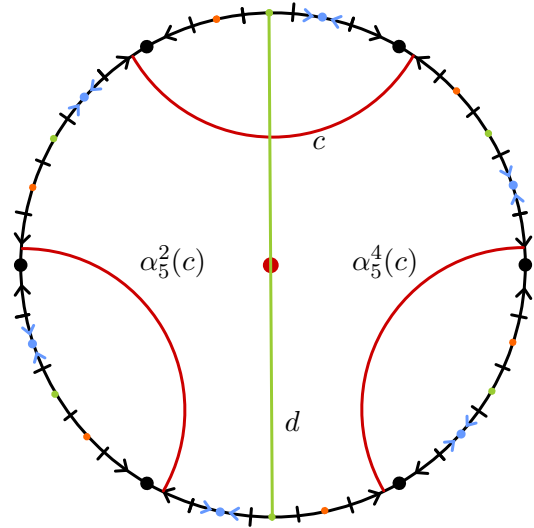

An individual basic chordal curve is not (in general) invariant under nontrivial powers of $\alpha_g$. Lemma \ref{lemma:commtrick} shows that nevertheless, by using both $C(\alpha_g^2)$ and $C(\alpha_g^3)$, the symmetry can be broken and the associated Dehn twists can be exhibited as elements of $\Gamma$. 

\begin{lem}\label{lemma:commtrick}
Let $c$ be a basic chordal curve of type $p$ or $r$. If $g \ge 3$, then $T_c \in \Gamma$. 
\end{lem}

\begin{proof}
If $c$ is of type $r$, then $c$ can be represented as a diameter of $D$ and hence $c \in C(\alpha_g^3) \le \Gamma$. Suppose now that $c$ is of type $p$, connecting edges $e_1, e_2$ of $\partial D$ of type $p$. In each of the model surfaces, if $g \ge 3$, then between $e_1$ and $e_2$ is an edge $e_3$ of type $r$ (see Figure \ref{figure:commtrick}). Consider the associated basic chordal curve $d$ of type $r$. As discussed above, $T_d \in \Gamma$, and also
\[
T_c T_{\alpha_g^2(c)} T_{\alpha_g^4(c)} \in C(\alpha_g^2) \le \Gamma. 
\]
By construction, the geometric intersection $i(c,d) = 1$, while also
\[
i(\alpha_g^2(c), d) = i(\alpha_g^4(c), d) = 0.
\]
 Thus, 
\[
(T_c T_{\alpha_g^2(c)} T_{\alpha_g^4(c)}) T_d (T_c T_{\alpha_g^2(c)} T_{\alpha_g^4(c)})^{-1}  = T_c T_d T_c^{-1} \in \Gamma.
\]

On the other hand, since $i(c,d) = 1$, the braid relation implies that 
\[
T_c T_d T_c^{-1} = T_d^{-1} T_c T_d \in \Gamma,
\]
and hence $T_c \in \Gamma$ as well.
\end{proof}

\subsection{\boldmath $D$-convexity}
The second piece of hyperbolic geometry we borrow is the notion of convexity. In the proof of Theorem \ref{main}, we will proceed inductively, showing that $\Gamma$ contains the mapping class groups for an increasing union of subsurfaces. In the inductive step, we will need to control the topology of the enlarged subsurface relative to the original; we accomplish this by restricting our attention to subsurfaces that are {\em convex} from the point of view of the hyperbolic metric on $\mathbb D^2$.

\begin{defn}[$D$-convex hull, $D$-convexity]
Let $\mathcal C = \{c_1, \dots, c_n\}$ be a collection of simple closed curves on $\Sigma_{g,1}$. Represent each $c_i$ as a union of chords on $D$, i.e. as a union of geodesics on the hyperbolic disk $\mathbb{D}^2$. The {\em $D$-convex hull} of $\mathcal C$ is the subsurface $\Hull(\mathcal C) \subseteq \Sigma_{g,1}$ constructed as follows: first, take a closed regular neighborhood of $\bigcup c_i$ (viewed as a subset of $D$), take the convex hull of this set in the hyperbolic metric on $\mathbb{D}^2$, project onto $\Sigma_{g,1}$, and then fill in any inessential boundary components. 

A subsurface $S \subset \Sigma_{g,1}$ is said to be {\em $D$-convex} if it can be represented as a convex region on $D$ with respect to the hyperbolic metric on $\mathbb{D}^2$.
\end{defn}

\begin{lem}\label{lemma:hull}
Let $\mathcal C_g$ denote the set of basic chordal curves of type $p$ and $r$ on $\Sigma_{g,1}$. Then $\Hull(\mathcal C_g) = \Sigma_{g,1}$ for all $g \ge 2$.
\end{lem}

\begin{proof}
For $g \equiv 1 \pmod 3$ this is clear from inspection of Figure \ref{figure:models}, since there are no edges of type $q$ at all; the case $g = 3$ similarly follows by inspection of Figure \ref{figure:genus3model}.  For $g \equiv 2 \pmod 3$, this is best seen by inspecting Figure \ref{figure:hull}. Here one must observe that the remaining boundary components $d_1$ and $d_2$ are in fact also both inessential and hence are filled in when constructing $\Hull(\mathcal C_g)$. For $g >3$ and $g \equiv 0 \pmod 3$, there is also exactly one family of edges of type $q$, and the same considerations as in the case $g \equiv 2 \pmod 3$ apply here as well.
\end{proof}

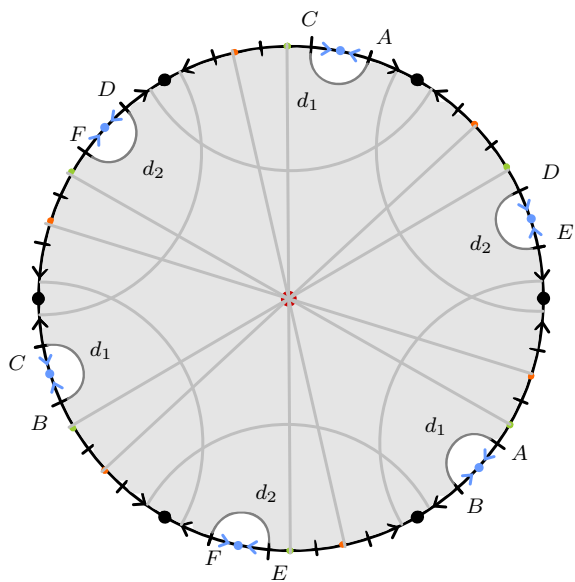
\begin{figure}[ht]
\labellist
\tiny
\pinlabel $A$ [bl] at 127.55 192.74
\pinlabel $A$ [tl] at 178.57 45.35
\pinlabel $B$ [tl] at 161.56 25.51
\pinlabel $B$ [tr] at 8.50 56.69
\pinlabel $C$ [tr] at 0.00 79.36
\pinlabel $C$ [b] at 104.87 198.41
\pinlabel $D$ [br] at 34.01 172.90
\pinlabel $D$ [bl] at 189.90 141.72
\pinlabel $E$ [bl] at 195.57 119.04
\pinlabel $E$ [tl] at 87.87 0.00
\pinlabel $F$ [tl] at 63.03 5.67
\pinlabel $F$ [bl] at 12.01 155.89
\pinlabel $d_1$ [tr] at 110.54 178.57
\pinlabel $d_1$ [bl] at 19.84 73.69
\pinlabel $d_1$ [br] at 158.73 45.35
\pinlabel $d_2$ [tr] at 175.73 124.71
\pinlabel $d_2$ [tl] at 39.68 153.06
\pinlabel $d_2$ [bl] at 82.20 19.84
\endlabellist
\includegraphics[scale=1]{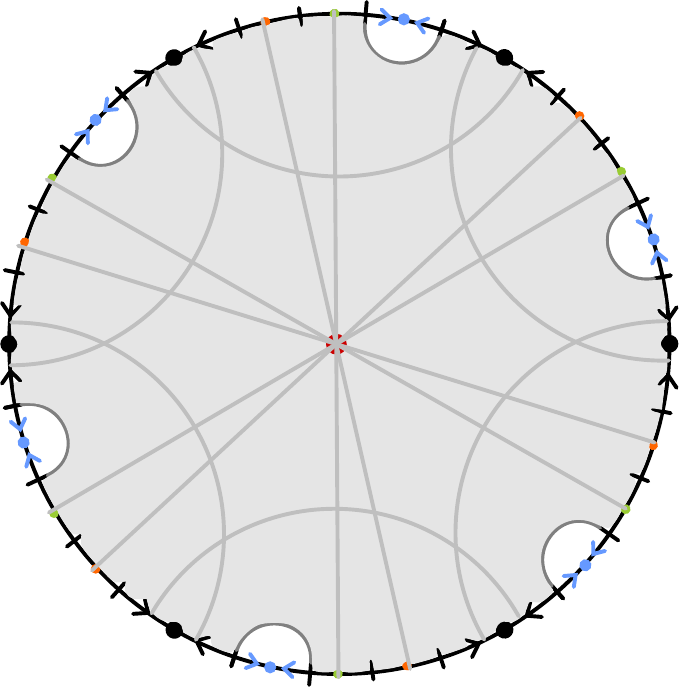}
\caption{The final step in constructing $\Hull(\mathcal C_g)$ for $g = 5$. The boundary components $d_1$ and $d_2$ are both inessential in $\Sigma_{g,1}$ and so are filled in when constructing $\Hull(\mathcal C_g)$.}
\label{figure:hull}
\end{figure}

\section{proof of Theorem \ref{main}} \label{section:proof}
We prove Theorem \ref{main} in Section \ref{subsection:proof}. The argument is inductive: we construct a sequence $S_0 \subset S_1 \subset \dots \subset S_k = \Sigma_{g,1}$ of subsurfaces and show that $\Mod(S_i) \le \Gamma$ for $i = 1, \dots, k$. The inductive step is fairly simple and relies on the notion of a ``stabilization'' of subsurfaces to be discussed in Section \ref{section:stab}. We consider separate base cases for the regimes $g \ge 4, g = 3$, and $g = 2$; these arguments are deferred to Sections \ref{section:ge4} -- \ref{section:2}.

\subsection{Stabilizations}\label{section:stab}

\begin{defn}[Stabilization]
Let $S \subset \Sigma$ be a subsurface, and let $c \subset \Sigma$ be a simple closed curve such that $c \cap S$ is a single arc (the endpoints of $c$ do not necessarily lie on distinct boundary components of $S$). The {\em stabilization of $S$ along $c$} is the subsurface $S^+$ constructed as a regular neighborhood of $S \cup c$ inside $\Sigma$.
\end{defn}

Stabilizations are useful because they allow for simple inductive generating sets for the associated mapping class groups. 

\begin{lem}[Stabilization]\label{prop:stab}
Let $S \subset \Sigma$ be a subsurface of genus at least $2$, and let $S^+$ denote the stabilization of $S$ along the simple closed curve $c$. Then
\[
\Mod(S^+) = \pair{T_c, \Mod(S)}.
\]
\end{lem}
\begin{proof}
There are two cases to consider: either $c$ enters and exits $S$ via the same boundary component, or else it enters along one component of $\partial S$ and exits along a distinct component. In the former, $S^+$ has the same genus as $S$ but gains an additional boundary component, and in the latter, $S^+$ has genus $g(S) + 1$ but one fewer boundary component. In either case, the change--of--coordinates principle for $S$ implies that $c$ can be extended to a configuration of curves $c_0 = c, \dots, c_n$ such that $c_i \subset S$ for $i >0$ and such that the associated twists generate $\Mod(S^+)$. For instance, one can take $c_0, \dots, c_n$ to be the Humphries generating set for $S^+$ and $c_1, \dots, c_n$ to be the Humphries generating set for $S$, so long as $g(S) \ge 2$. The result follows. \end{proof}

\subsection{Proof of Theorem \ref{main}}\label{subsection:proof}

The result for $g = 2$ will be established by separate methods in Section \ref{section:2}; we therefore assume $g \ge 3$. We will express $\Sigma_{g,1}$ as a sequence
\[
S_0 \subset S_1 \subset \dots \subset S_k = \Sigma_{g,1}
\]
of stabilizations of $S$ along curves in the set $\mathcal C_g$ of basic chordal curves of type $p$ and $r$. At each stage we will see that $\Mod(S_i) \le \Gamma$. 

\para{The base case} In Sections \ref{section:ge4} and \ref{section:3}, we will establish the following lemma.
\begin{lem}\label{prop:subsurface}
For each of the model surfaces shown in Figure \ref{figure:models} (excluding $g = 2$) as well as the model surface in genus $3$ shown in Figure \ref{figure:genus3model}, there is a $D$-convex subsurface $S_0 \subset \Sigma_{g,1}$ of genus $2$ such that $\Mod(S_0) \le \Gamma$.
\end{lem}

\para{The inductive step} Suppose that $S_i$ is given as a $D$-convex subsurface with $\Mod(S_i) \le \Gamma$. Suppose first that every curve $c_j \in \mathcal C_g$ is contained in $S_i$. Since $S_i$ is $D$-convex and the $D$-convex hull of $\mathcal C_g$ is $\Sigma_{g,1}$ by Lemma \ref{lemma:hull}, in this case $S_i = \Sigma_{g,1}$ and the theorem is proved.

Otherwise, select $c_j \in \mathcal C_g$ a curve {\em not entirely contained} in $S_i$. We then define $S_{i+1}$ to be the $D$-convex hull of $S_i \cup c_j$. Since $S_i$ is $D$-convex and $c_j$ is a basic chordal curve, necessarily $c_j$ enters and exits $S_i$ exactly once, and hence $S_{i+1}$ is the stabilization of $S_i$ along $c_j$. By Lemma \ref{lemma:commtrick}, $T_{c_j} \in \Gamma$, and by hypothesis, $\Mod(S_i) \le \Gamma$. By the stabilization lemma (Lemma \ref{prop:stab}), therefore $\Mod(S_{i+1}) \le \Gamma$ as well.\qed

\subsection{\boldmath Proof of Lemma \ref{prop:subsurface} for $g \ge 4$}\label{section:ge4} 

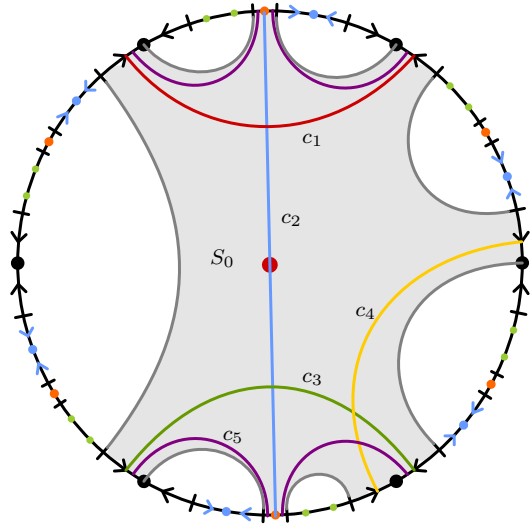
\begin{figure}[ht]
\labellist
\tiny
\pinlabel $c_1$ [br] at 120 140
\pinlabel $c_2$ [bl] at 100 110
\pinlabel $c_3$ [br] at 120 50
\pinlabel $c_4$ [br] at 140 75
\pinlabel $c_5$ [tr] at 90 37
\pinlabel $S_0$ at 80 100
\endlabellist
\includegraphics[scale=1]{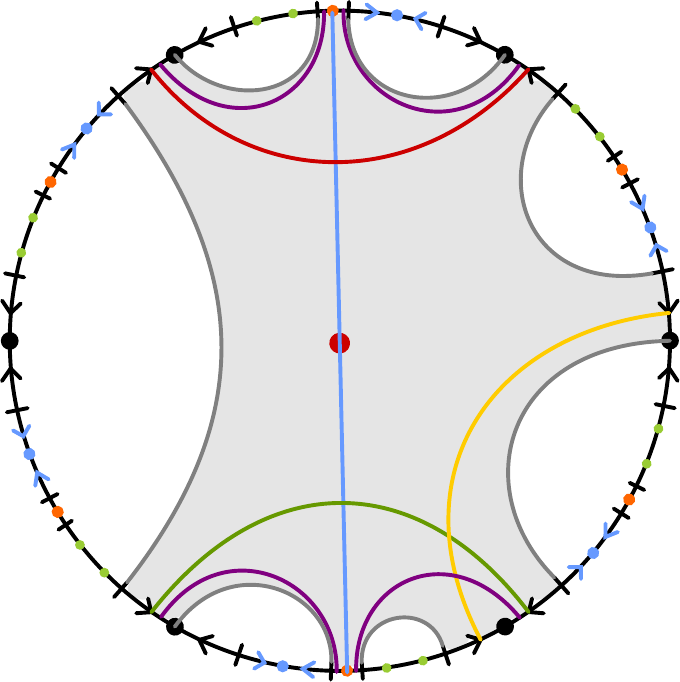}
\caption{For $g \equiv 1 \pmod 3$, the surface $S_0$ is taken as the $D$-convex hull of the curves $c_1, \dots, c_5$ as shown.}
\label{figure:1mod3}
\end{figure}

We will describe a chain of five curves $c_1, \dots c_5$ such that $c_1 \cup c_3 \cup c_5$ bounds a pair of pants. Such a configuration is supported on a surface $S_0$ of genus $2$ with one boundary component, and the associated Dehn twists form the Humphries generating set for $\Mod(S_0)$. The curves we will describe will either be of type $p$ (hence in $\Gamma$ by Lemma \ref{lemma:commtrick}) or else invariant under $\alpha_g^3$ (and hence in $C(\alpha_g^3) \le \Gamma$). Such a configuration is illustrated in Figure \ref{figure:1mod3} in the case of $g \equiv 1 \pmod 4$, but the construction we describe below works on all the model surfaces. 

Let $c_1$ be the curve of type $p$ connecting the edges of type $p$ labeled $1$ in Figure \ref{figure:models}. We take $c_3 = \alpha_g^3(c_1)$ and $c_4 = \alpha_g^4(c_1)$. Let $c_2$ be the curve of type $r$ intersecting $c_1$ and $c_3$. Finally, let $c_5$ be the curve obtained by connect-summing $c_1$ and $c_3$ along one of the segments of $c_2$. As shown in Figure \ref{figure:1mod3}, $c_5$ is invariant under $\alpha_g^3$ as required. 

We find that $c_1, c_3, c_4$ are curves of type $p$, and $c_2, c_5$ are invariant under $\alpha_g^3$, so all associated twists are elements of $\Gamma$, and hence $\Mod(S_0) \le \Gamma$ as claimed. \qed


\subsection{\boldmath Proof of Lemma \ref{prop:subsurface} for $g = 3$}\label{section:3}
\begin{figure}[ht]
\labellist
\tiny
\pinlabel $c_1$ [br] at 104.71 130.38
\pinlabel $c_2$ [bl] at 110.54 39.68
\pinlabel $c_3$ [bl] at 51.02 59.52
\pinlabel $c_4$ [b] at 19.84 104.87
\pinlabel $c_5$ [tl] at 39.68 99.20
\pinlabel $c_6$ [tr] at 82.20 113.38
\pinlabel $1$ [bl] at 155.89 175.73
\pinlabel $6$ [bl] at 136.05 187.07
\pinlabel $6$ [br] at 59.52 187.07
\pinlabel $5$ [br] at 39.68 175.73
\pinlabel $5$ [br] at 0.00 107.71
\pinlabel $4$ [tr] at 0.00 85.03
\pinlabel $4$ [tr] at 39.68 19.84
\pinlabel $3$ [tr] at 59.52 8.50
\pinlabel $3$ [tl] at 136.05 8.50
\pinlabel $2$ [tl] at 158.73 19.84
\pinlabel $2$ [tl] at 195.57 85.03
\pinlabel $1$ [bl] at 195.57 110.54
\pinlabel $x$ [tr] at 144.55 175.73
\pinlabel $S_0'$ at 80 80
\pinlabel $S_0$ at 290 80
\endlabellist
\includegraphics[scale=1]{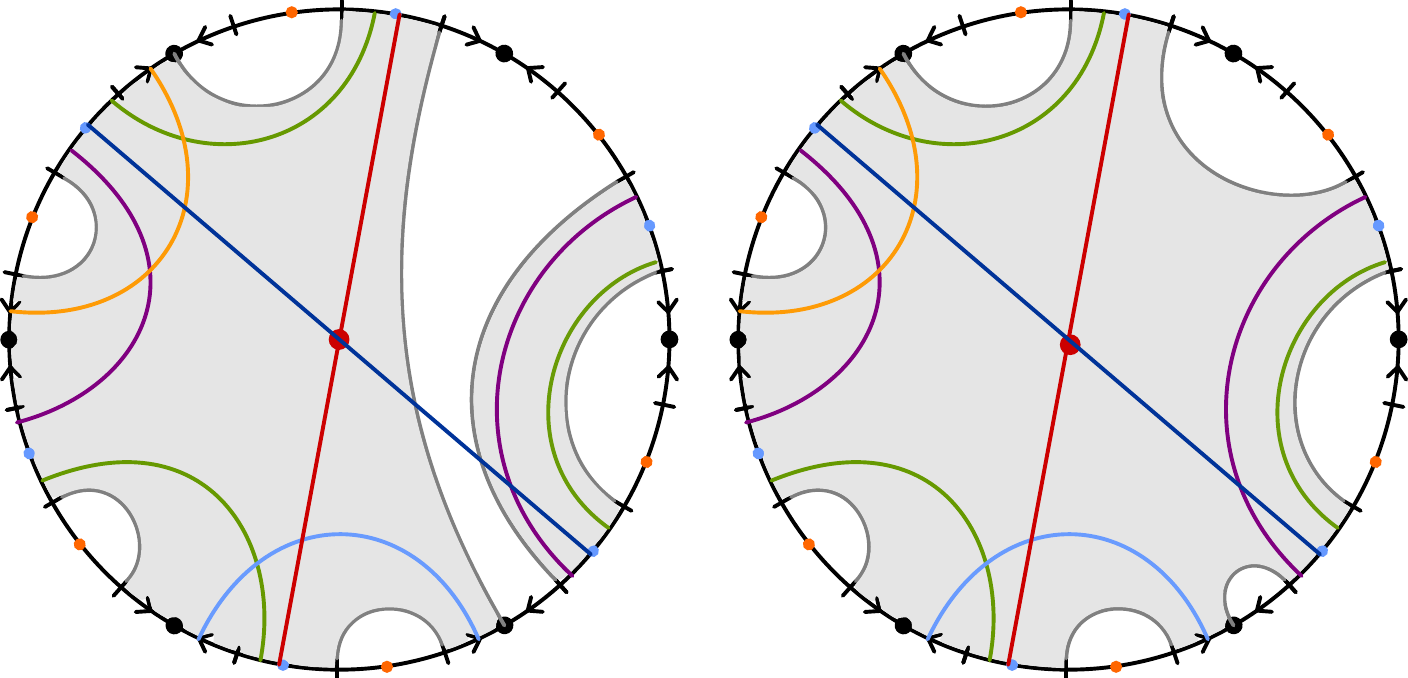}
\caption{The model surface for $g=3$, built from the monodromy tuple $1\ 5\ 3^2$. The surface $S'_0$ is taken as a regular neighborhood of the curves $c_1, \dots, c_5$ as shown, and then $S_0$ is the $D$-convex hull of $c_1, \dots, c_6$.}
\label{figure:genus3model}
\end{figure}
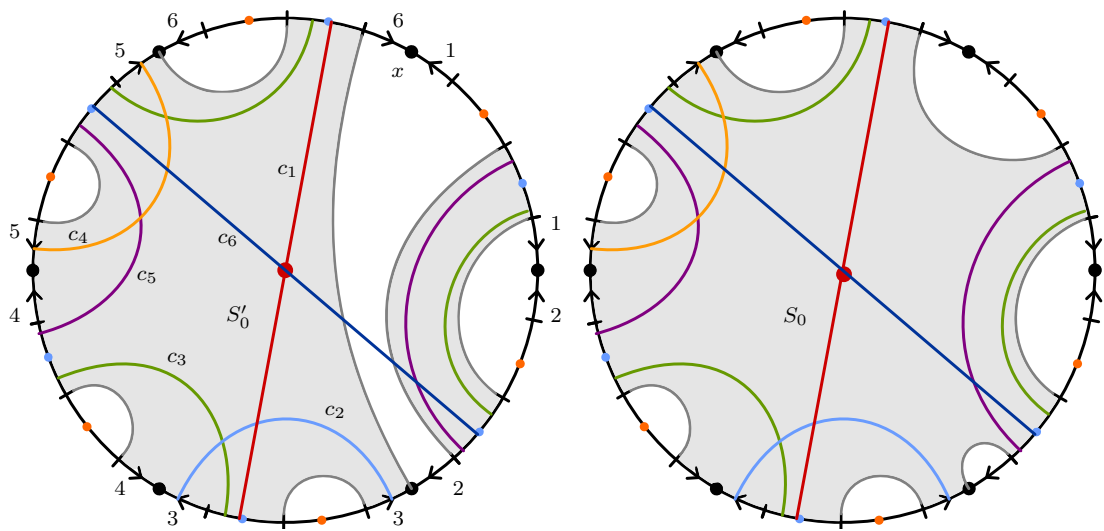

Recall from \eqref{tuple} that the monodromy tuple for $g = 3$ is $1\ 5\ 3^2$. The model surface for this tuple is shown in Figure \ref{figure:genus3model}. To establish Lemma \ref{prop:subsurface} in this case, we first consider the subsurface $S'_0$ shown at left in Figure \ref{figure:genus3model}. By construction $S'_0$ is a regular neighborhood of $c_1, \dots, c_5$. Observe that $c_1$ and $c_5$ are $\alpha_g^3$-invariant, $c_3$ is $\alpha_g^2$-invariant, and $c_2$ and $c_4$ are basic chordal curves of type $p$. Thus each associated Dehn twist is an element of $\Gamma$. As above, $T_{c_1}, \dots, T_{c_5}$ determines the Humphries generating set for $S'_0$, and we conclude that $\Mod(S'_0) \le \Gamma$. 

We next consider $S_0$. By construction, $S_0$ is the $D$-convex hull of $c_1, \dots, c_6$, and it is also clear that $S_0$ is the stabilization of $S'_0$ along $c_6$. Since $c_6$ is a basic chordal curve of type $r$, we have $T_{c_6} \in \Gamma$ by Lemma \ref{lemma:commtrick}. By the stabilization lemma (Lemma \ref{prop:stab}), it follows that $\Mod(S_0) \le \Gamma$ as required.\qed

\subsection{\boldmath Proof of Theorem \ref{main} for $g = 2$}\label{section:2}

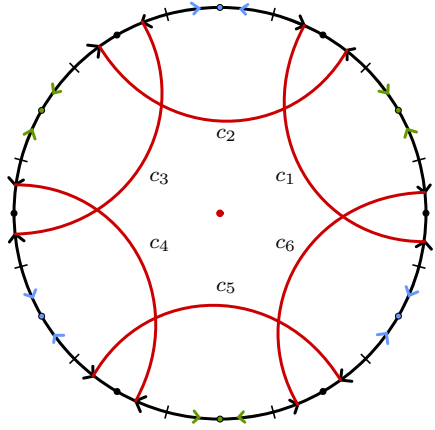
\begin{figure}[ht]
\labellist
\tiny
\pinlabel $c_1$ [tr] at 113.38 99.20
\pinlabel $c_2$ [t] at 85.03 116.21
\pinlabel $c_3$ [tl] at 53.85 99.20
\pinlabel $c_4$ [bl] at 53.85 65.19
\pinlabel $c_5$ [b] at 85.03 48.18
\pinlabel $c_6$ [br] at 113.38 65.19
\endlabellist
\includegraphics[scale=1]{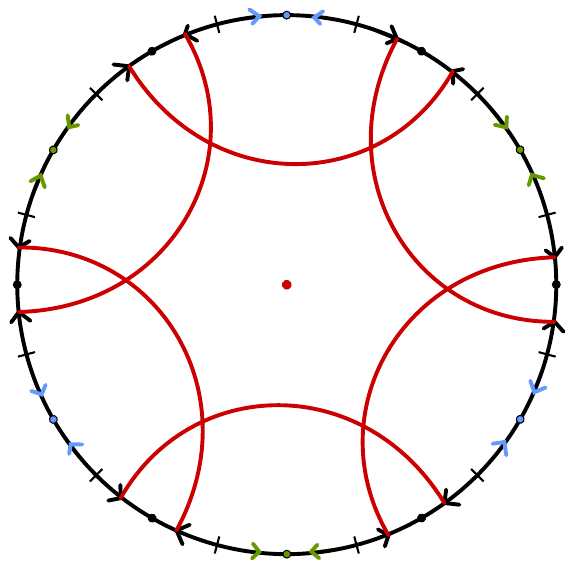}
\caption{The model surface for $g = 2$ and the curves $c_1, \dots, c_6$.}
\label{figure:2}
\end{figure}

For $g = 2$, we take a different approach based around an explicit factorization of $\alpha_2$ into Dehn twists. The model for $g = 2$ is shown in Figure \ref{figure:2}. For ease of notation, we write $T_i$ in place of $T_{c_i}$ throughout the argument. The mapping class group is generated by the twists $T_{i}$ for $i = 1, \dots, 5$; we will show that all $T_{i} \in \Gamma$. The fundamental observation is that
\[
\alpha_2 = T_{1} T_2T_3T_4T_5.
\]
We also observe that 
\[
T_1T_4,\ T_2T_5,\ T_3T_6 \in \Gamma
\]
since these pairs of curves are invariant under $\alpha_2^3$, and also
\[
T_1 T_3 T_5,\ T_2 T_4 T_6 \in \Gamma
\]
since these triples are invariant under $\alpha_2^2$. 

We consider the expression of elements of $\Gamma$
\begin{align*}
\alpha_2 (T_2 T_5)^{-1} (T_1 T_4)^{-1} &= T_1 T_2 T_3 T_4 T_5 T_5^{-1} T_2^{-1} T_4^{-1} T_1^{-1}\\
	& = T_1 T_2 T_3 T_2^{-1} T_1^{-1},
\end{align*}
with the second equality holding by the commutativity of $T_i$ and $T_j$ whenever $i \ne j \pm 1$. Conjugating $T_1 T_2 T_3 T_2^{-1} T_1^{-1}$ by $(T_1 T_3 T_5)^{-1}$ shows that the element
\[
(T_1 T_3 T_5)^{-1} T_1 T_2 T_3 T_2^{-1} T_1^{-1} (T_1 T_3 T_5) = T_3^{-1} T_2 T_3 T_2^{-1} T_3
\]
is also in $\Gamma$. Conjugating this by $(T_3 T_6)$ shows that 
\[
T_2 T_3 T_2^{-1} \in \Gamma;
\]
a final conjugation by $(T_2 T_5)^{-1}$ reveals that $T_3 \in \Gamma$. Conjugation by $\alpha_2$ now exhibits all $T_i$ in $\Gamma$.  \qed

\section{Mapping class group actions on $\R^3$}\label{section:R3}

In this final section, we show how Theorem \ref{main} implies Theorem \ref{R3} and Corollary \ref{R23}. Recall that the objective is to show that for $g \ge 4$, any action of $\Mod(\Sigma_{g,1})$ on $\R^3$ has a globally-invariant line, and consequently that $\Mod(\Sigma_{g,1})$ does not act by $C^1$ diffeomorphisms on $\R^2$ or $\R^3$.

We consider an action $\rho$ of $\Mod(\Sigma_{g,1})$ on $\mathbb{R}^3$. Recall that any such action must necessarily preserve orientation. We will appeal to the following result of Lanier--Margalit \cite[Theorem 1.1]{LanierMargalit}.
\begin{thm}[Lanier--Margalit]\label{theorem:LM}
For $g \ge 3$, every nontrivial periodic mapping class that is not hyperelliptic normally generates $\Mod(\Sigma_{g,1})$ or $\Mod(\Sigma_{g})$.
\end{thm}
We remark that \cite[Theorem 1.1]{LanierMargalit} only discusses the case $\Mod(\Sigma_{g})$, however the same method applies to $\Mod(\Sigma_{g,1})$.  
\vskip 0.3cm

By Theorem \ref{theorem:LM}, $\alpha_g^2$ normally generates $\Mod(\Sigma_{g,1})$. Therefore if $\rho$ is not trivial, then $\rho(\alpha_g^2)$ is not trivial. By local Smith theory (\cite[Theorem 20.1]{Bredon}), the fixed point set of $\alpha_g^2$ is a $\Z/3\Z$--homology manifold of dimension less than $2$. By \cite[Theorem 16.32]{Bredon}, when the dimension of a homology manifold is less than $2$, it is a topological manifold. We claim that the fixed set $F(\alpha_g^2)$ of $\rho(\alpha_g^2)$ is a single topological line in $\R^3$. Since $\R^3$ is acyclic, also $F(\alpha_g^2)$ is also acyclic (c.f. \cite[Corollary 19.8]{Bredon}). Hence $F(\alpha_g^2)$ must have exactly one component. By \cite[Corollary 19.11]{Bredon}, $F(\alpha_g^2)$ is a line, since we can consider the action on the one-point compactification of $\R^3$. This can be compared with the fact that the fixed set of a torsion element in $SO(3)$ is a single line in $\R^3$.

Since $\alpha_g^3$ is not hyperelliptic, the same argument shows that the fixed set of $\rho(\alpha^3)$ is also a line $F(\alpha_g^3)$. We claim that 
$F(\alpha_g^2)=F(\alpha_g^3)$, which will be denoted by $F$. If these lines are distinct, then the action of $\rho(\alpha_g)$ on $F(\alpha_g^3)$ must be nontrivial (otherwise $\rho(\alpha_g^2)$ would act trivially as well, implying $F(\alpha_g^2) = F(\alpha_g^3)$). As $\rho(\alpha_g^3)$ acts trivially on $F(\alpha_g^3)$ by construction, it follows that $\rho(\alpha_g)$ acts as an element of order $3$. This is a contradiction: there is no nontrivial action of $\Z/3\Z$ on a line. Thus by Theorem \ref{main}, $\Mod(\Sigma_{g,1})$ must preserve $F$, establishing Theorem \ref{R3}.

Now suppose $\rho$ acts by $C^1$ diffeomorphisms, and let $x \in F$ be any fixed point. Taking derivatives at $x$, we obtain a representation $R: \Mod(\Sigma_{g,1})\to \GL(3,\R)$. According to \cite[Theorem 1.1]{FranksHandel}, any such homomorphism is trivial. The Thurston stability theorem \cite{Thurston} then implies that the image of $\rho$ must be locally-indicable, i.e.  every finitely-generated subgroup admits a surjection onto $\Z$. In particular, $\text{im}(\rho)$ must be torsion-free, and so $\rho(\alpha_g)$ is the identity map. By Theorem \ref{theorem:LM}, it follows that the entire representation $\rho$ is trivial.\qed

\begin{rem}
In fact, the conclusions of Theorem \ref{R3} and Corollary \ref{R23} hold for $g =3$ as well, using slightly different arguments. We briefly discuss this. From the discussion above, if $\rho(\alpha_g^3)$ is not trivial, the same arguments apply. Otherwise, denote by $H: \Mod(S_{g,1})\to \text{Sp}(2g,\Z)$ the induced action on $H_1(S_g;\Z)$. If $\alpha_g^3$ is hyperelliptic and $\rho(\alpha_g^3)$ is the identity, we claim that $\rho$ factors through $H$. This is because the hyperelliptic involution $\alpha_g^3$ normally generates the group $H^{-1}(\pm I)$ by \cite[Proposition 3.3]{LanierMargalit}, whose proof also works for the punctured case. 

To conclude, we claim that there is no action of $\text{Sp}(2g,\Z)$ on $\mathbb{R}^3$, even by homeomorphisms. According to \cite[Corollary 1]{Zimmermann}, $(\Z/3\Z)^3$ is not a subgroup of $\text{Homeo}(\mathbb{R}^3)$. The claim then follows from \cite[Lemma 10]{ChenLanier}.
\end{rem}

\bibliography{citing}{}
\end{document}